\documentclass[11pt]{amsart}

\usepackage{a4wide,amsmath,amsfonts,amssymb,amsthm,mathrsfs,bbm}
\usepackage[hyperfootnotes=false]{hyperref}
\usepackage[dvipsnames]{xcolor}
\usepackage[shortlabels]{enumitem}
\allowdisplaybreaks[2]

\newtheorem{theorem}{Theorem}[section]
\newtheorem{lemma}[theorem]{Lemma}

\newtheorem{remark}[theorem]{Remark}
\newtheorem{proposition}[theorem]{Proposition}

\newtheorem{example}[theorem]{Example}
\newtheorem{assumption}[theorem]{Assumption}

\numberwithin{equation}{section}
\allowdisplaybreaks

\renewcommand{\d}{\mathrm{d}}
\newcommand{\dd}{\,\mathrm{d}}
\newcommand{\R}{\mathbb{R}}
\newcommand{\N}{\mathbb{N}}
\newcommand{\E}{\mathbb{E}}
\renewcommand{\P}{\mathbb{P}}
\newcommand{\bP}{\mathbb{P}}
\newcommand{\PP}{\mathbb{P}}
\newcommand{\W}{\mathcal{W}}
\renewcommand{\d}{\mathrm{d}}
\renewcommand{\epsilon}{\varepsilon}
\renewcommand{\P}{\mathbb{P}}
\newcommand{\cF}{\mathcal{F}}
\newcommand{\cP}{\mathcal{P}}
\def\cL{\mathscr L}

\title[Non-uniqueness for McKean--Vlasov SDEs with H\"older diffusion coefficients]{Non-uniqueness for McKean--Vlasov SDEs with H\"older diffusion coefficients}

\author[Bergerhausen]{Martin Bergerhausen}
\address{Martin Bergerhausen, University of Mannheim, Germany}
\email{martin.bergerhausen@uni-mannheim.de}

\author[Nikolaev]{Paul Nikolaev}
\address{Paul Nikolaev, Technical University of Berlin, German}
\email{paul.nikolaev@tu-berlin.de}

\author[Pr{\"o}mel]{David J. Pr{\"o}mel}
\address{David J. Pr{\"o}mel, University of Mannheim, Germany}
\email{proemel@uni-mannheim.de}

\date{\today}

\begin{document}

\begin{abstract}
  We show that uniqueness may fail for one-dimensional McKean--Vlasov stochastic differential equations with H\"older continuous diffusion coefficients. More precisely, for every \(\alpha\in(1/2,1)\), we construct an equation with a diffusion coefficient that is \(\alpha\)-H\"older continuous in the state variable, Lipschitz continuous in the measure variable, and of linear growth, for which uniqueness in law and pathwise uniqueness fail. The construction relies on a spreading estimate together with a Schauder fixed-point argument.
\end{abstract}

\maketitle





\section{Introduction}

McKean--Vlasov stochastic differential equations (McKean--Vlasov SDEs), also known as mean-field or distribution-dependent SDEs, extend ordinary stochastic differential equations by allowing the coefficients to depend not only on the current state of the solution process, but also on its law. Since the seminal works of Kac~\cite{Kac56}, McKean~\cite{McKean66}, and Vlasov~\cite{Vlasov1968}, McKean--Vlasov SDEs have developed into a central topic in probability theory, with close connections to nonlinear partial differential equations, and play an important role, among other areas, in the study of interacting particle systems, mean-field limits, and propagation of chaos; see, e.g., \cite{Sznitman1991,Jabin2017,Carmona2018} for comprehensive overviews. A fundamental problem is to determine regularity conditions on the coefficients of McKean--Vlasov SDEs that guarantee existence of a solution and its uniqueness, either pathwise or in law.

Under Lipschitz continuity of the coefficients in the state and measure variables, strong well-posedness of McKean--Vlasov SDEs is classical; see, e.g., \cite{Sznitman1991}. A substantial literature is devoted to relaxing these assumptions in several directions. One approach replaces global Lipschitz continuity by weaker structural conditions, such as monotonicity, Osgood-type or Lyapunov conditions; see, e.g., \cite{Wang2018,Bahlali2020,Hammersley2021,Kalinin2024b}. Another important line of research exploits the regularizing effect by noise in case of non-degenerate diffusion coefficients, which can yield weak or strong uniqueness under considerably weaker regularity assumptions; see, e.g., \cite{Huang2019,Mehri2019,Mishura2020,Veretennikov2021,Pascucci2025}. Particularly relevant to the present work is the classical result of Yamada and Watanabe~\cite{Yamada1971}, which states that for one-dimensional ordinary SDEs, pathwise uniqueness holds provided the drift coefficient is sufficiently regular and the diffusion coefficient is H\"older continuous with exponent at least \(1/2\). Yamada--Watanabe-type arguments have also been extended to McKean--Vlasov SDEs under various non-Lipschitz regularity assumptions on the diffusion coefficient; see, for instance, \cite{Bahlali2020,Bao2022,Promel_MF,Huang2023,Kalinin2024,Ning2026}. In the Brownian setting, these results typically impose comparatively strong regularity assumptions on the dependence of the diffusion coefficient on the law of the solution.

The purpose of the present paper is to show that the classical Yamada--Watanabe uniqueness threshold does not extend in full generality to McKean--Vlasov SDEs with Brownian noise. For every $\alpha\in(1/2,1)$, we construct a one-dimensional McKean--Vlasov SDE with Lipschitz continuous drift and a diffusion coefficient that is $\alpha$-H\"older continuous in the state variable, Lipschitz continuous in the measure variable with respect to the $1$-Wasserstein distance, and of linear growth, for which both uniqueness in law and pathwise uniqueness fail, see Theorem~\ref{thm:main}. More precisely, the presented counterexample is of the form
\begin{equation*}
   \dd X_t=\big(C_1|X_t|^\alpha+F(\cL(X_t))\big)\dd B_t,\qquad X_0 \sim\delta_0,
\end{equation*}
where $F$ is a Lipschitz functional of the solution's law $\cL(X_t)$ satisfying $F(\delta_0)=0$ and $F(\mu)\geq \kappa \|\mu\|_1$ for all probability distributions $\mu$ and some constant $\kappa > 0$, and $B=(B_t)_{t\in [0,T]}$ is a one-dimensional Brownian motion. Thus, spatial H\"older continuity of any order strictly larger than $1/2$, together with Lipschitz continuity in the law variable, does not by itself guarantee either pathwise uniqueness or uniqueness in law. Moreover, the diffusion coefficient is degenerate, highlighting the role of non-degeneracy in uniqueness results that rely on the regularizing effect of noise. Let us also remark that non-uniqueness caused by the dependence on the law had already been observed by Scheutzow~\cite{Scheutzow1987} for a class of Vlasov--McKean equations without Brownian diffusion.


\medskip
\noindent \textbf{Organization of the paper:} Section~\ref{sec:main result} introduces the setting and states the non-uniqueness result and Section~\ref{sec:proof} contains its proof. Appendix~\ref{sec:appendix} provides the proof of an auxiliary weak existence result.

\medskip
\noindent\textbf{Acknowledgments:} P. Nikolaev gratefully acknowledges funding by the German Research Foundation (DFG) through CRC/TRR 388 ``Rough Analysis, Stochastic Dynamics and Related Fields'', Project A11. D. J. Pr\"omel gratefully acknowledges his affiliation with the Department of Mathematics at King’s College London, United Kingdom. M. Bergerhausen gratefully acknowledges the hospitality of BI Norwegian Business School during a research visit, where part of this work was carried out.

\medskip
\noindent\textbf{AI Disclosure:} The mathematical content of this paper is entirely human-generated. AI tools were employed only for proofreading, grammatical corrections, and linguistic editing. 

\section{Main result: Non-uniqueness for McKean--Vlasov SDEs}\label{sec:main result}
 
Fix $T>0$ and a filtered probability space $(\Omega,\mathcal F,(\mathcal F_t)_{t \in [0,T]},\PP)$ satisfying the usual conditions and carrying a one-dimensional Brownian motion $B=(B_t)_{t\in [0,T]}$ as well as an $\cF_0$-measurable random variable $\xi$. Let $(\R,|\cdot|)$ be the Euclidean space and let $\cP_p(\R)$ denote the probability measures with $p$ finite moments, equipped with the $p$-Wasserstein metric $\W_p$, that is, for $\mu,\nu\in\mathcal{P}_p(\R)$ we define $\W_p(\mu,\nu)$ by
\begin{equation*}
  \W_p(\mu,\nu):= \inf_{\pi \in \Pi(\mu,\nu)} \Big[ \int_{\R \times \R} |x-y|^p \,\pi(\d x,\d y)\Big]^{1/p},
\end{equation*}
where $\Pi(\mu,\nu)$ denotes the set of probability measures on $\R\times \R$ with marginals $\mu$ and $\nu$, see \cite[Chapter~5]{Carmona2018} for more details, and $\|\mu\|_1:=\int_\R|y|\,\mu(\d y)$. We denote the law of some random variable $X$ by $\cL(X)$ and the Dirac measure at $0$ by $\delta_0$.

\smallskip

We consider the McKean--Vlasov stochastic differential equation (McKean--Vlasov SDE)
\begin{equation}\label{eq:MV}
  \dd X_t=b(t,X_t,\mu_t)\dd t+\sigma(t,X_t,\mu_t)\dd B_t,\qquad
  \mu_t=\cL(X_t),\qquad X_0 =\xi \sim\mu_0,
\end{equation}
where $b, \sigma \colon [0,T] \times \R \times \cP_1(\R) \to \R$ are measurable functions and $\mu_0 \in \cP_2(\R)$. The integral w.r.t. the Brownian motion $B$ is defined as a stochastic It{\^o} integral. For general introductions to McKean--Vlasov SDEs refer to \cite{Sznitman1991,Jabin2017,Carmona2018}.

\smallskip

Let us recall the concepts of strong and weak solutions and pathwise uniqueness. We call a continuous stochastic process $(X_t)_{t \in [0,T]}$ with $\sup_{t \in [0,T]} \E[|X_t|^p]< \infty$ on the given filtered probability space $(\Omega,\mathcal F,(\mathcal F_t)_{t \in [0,T]},\PP)$ a \emph{strong $L^p$-solution} of the McKean--Vlasov SDE~\eqref{eq:MV} if it is adapted to the usual augmented filtration of $\sigma(\xi,B_s\colon 0 \leq s \leq t)$, ${t \in [0,T]}$, the integral equation~\eqref{eq:MV} holds $\P$-almost surely, and $\int_0^t |b(s,X_s,\mu_s)| + |\sigma(s,X_s,\mu_s)|^2 \dd s < \infty$ $\P$-almost surely for all $t \in [0,T]$. A \emph{weak solution} of $\eqref{eq:MV}$ is a triple $(X,B)$, $(\Omega,\mathcal F,\PP)$, $(\mathcal F_t)_{t \in [0,T]}$, such that
\begin{itemize}
  \item $(\Omega,\mathcal F,\PP)$ is a probability space, $(\mathcal F_t)_{t \in [0,T]}$ is a filtration of sub-$\sigma$-algebras of $\mathcal F$ satisfying the usual conditions,
  \item $X = (X_t)_{t \in [0,T]}$ with $\sup_{t \in [0,T]} \E[|X_t|]< \infty$ is an $(\mathcal F_t)_{t \in [0,T]}$-adapted, continuous process, $B = (B_t)_{t \in [0,T]}$ is a Brownian motion with respect to $(\mathcal F_t)_{t \in [0,T]}$,
  \item $\int_0^t |b(s,X_s,\mu_s)| + |\sigma(s,X_s,\mu_s)|^2 \dd s < \infty$ $\P$-a.s. for any $t \in [0,T]$, and
  \item the integral equation in \eqref{eq:MV} holds $\P$-a.s. for $(X,B)$ on $(\Omega,\mathcal F,\PP)$ and $X_0 \sim\mu_0$.
\end{itemize}
We say \emph{uniqueness in law} holds for the McKean--Vlasov SDE~\eqref{eq:MV} if any two weak solutions possess the same law. Moreover, we say \emph{pathwise uniqueness} holds for the McKean--Vlasov SDE~\eqref{eq:MV} if for every two weak solutions $(X,B)$, $(\Omega,\mathcal F,\PP)$, $(\mathcal F_t)_{t \in [0,T]}$ and $(\tilde X,B)$, $(\Omega,\mathcal F,\PP)$, $(\tilde{\mathcal F}_t)_{t \in [0,T]}$ with common Brownian motion and with common initial value, i.e. $\PP[X_0 = \tilde{X}_0]=1$, the two processes $X$ and $\tilde X$ are indistinguishable.

\begin{assumption}\label{ass:main}
  There are $L, C_\sigma>0$ and $\alpha\in(\tfrac12,1]$ such that for all $t \in [0,T]$, $x,y \in \R$ and $\mu,\nu \in \cP_1(\R)$ we have:
  \begin{enumerate}[label=\textup{(A\arabic*)},leftmargin=3.2em]
    \item $|b(t,x,\mu)-b(t,y,\nu)|\le L\bigl(|x-y|+\W_1(\mu,\nu)\bigr)$ and $|b(t,x,\mu)|\le L\bigl(1+|x|+\|\mu\|_1\bigr)$,
    \item\label{A2} $|\sigma(t,x,\mu)-\sigma(t,y,\nu)|\le C_\sigma \bigl(|x-y|^\alpha + \W_1(\mu, \nu)\bigr)$, together with linear growth condition $|\sigma(t,x,\mu)|\le C_\sigma(1+|x|+\|\mu\|_1)$. 
  \end{enumerate}
\end{assumption}

The standard continuity and growth assumptions postulated in Assumption~\ref{ass:main} guarantee the existence of solutions to the McKean--Vlasov SDE~\eqref{eq:MV}.

\begin{remark}\label{prop:existence}
  Under Assumption~\ref{ass:main} the McKean--Vlasov SDE~\eqref{eq:MV} admits a weak solution $(X,B)$, $(\Omega,\mathcal F,\PP)$, $(\mathcal F_t)_{t \in [0,T]}$ with $\sup_{t \in [0,T]} \E[|X_t|^2]< \infty$, which follows by \cite[Theorem~2.10]{Hammersley2021}, as verified in Appendix~\ref{sec:appendix} below.

  If we have $\alpha = 1$ in Assumption~\ref{ass:main} and restrict the coefficients to $\cP_2(\R)$, then the diffusion coefficient is Lipschitz in law and space and McKean--Vlasov SDE~\eqref{eq:MV} has a strong solution and pathwise uniqueness holds, see \cite[Theorem~4.21]{Carmona2018}.
\end{remark}

However, for $\alpha\in(\frac12,1)$ uniqueness in law and pathwise uniqueness fail in general under Assumption~\ref{ass:main}, as presented in the following theorem.

\begin{theorem}\label{thm:main}
  Let $F\colon \cP_1(\R) \to [0,\infty)$ satisfy $F(\delta_0)=0$, $|F(\mu)-F(\nu)|\leq L_F \W_1(\mu,\nu)$, and $F(\mu)\geq \kappa \|\mu\|_1$ for all $\mu,\nu \in \cP_1(\R)$ and for some $\kappa, L_F > 0$. Let $\alpha\in(\tfrac12,1)$, $C_1>0$, define on $[0,T]\times\R\times\cP_1(\R)$ the coefficients
  \begin{equation}\label{eq:McKV_counterexample}
    b\equiv0\quad \text{and}\quad
    \sigma(t,x,\mu):=C_1|x|^{\alpha}+F(\mu),
  \end{equation}
  and consider \eqref{eq:MV} with initial condition $\mu_0:=\delta_0$. Then, the following holds:
  \begin{enumerate}[label=\textup{(\roman*)}]
    \item $b$ and $\sigma$ satisfy Assumption~\ref{ass:main} with H\"older exponent $\alpha$. Consequently, for every $p \geq 1$, the restriction of $\sigma$ to $\cP_p(\R)$ is Lipschitz continuous with respect to $\W_p$.
    \item $X\equiv0$ solves the McKean--Vlasov SDE~\eqref{eq:MV}, with flow $\mu_t\equiv\delta_0$.
    \item There exists a strong $L^2$-solution $Y^{*}$ of \eqref{eq:MV} that satisfies the inequality 
    \begin{equation*}
      \E[|Y^{*}_t|]\;\ge\; c_0 t^{\gamma},\qquad t\in[0,T_0],
    \end{equation*}
    for some $T_0=T_0(L_F,T)\in(0,T]$, $\varepsilon_0=\varepsilon_0(\alpha,C_1,\kappa)>0$, both explicit, $c_0:=\frac{\min\{\epsilon_0,1\}}{L_F}>0$, and $\gamma=\frac{\alpha}{1-\alpha}$.
    In particular, its flow $(\nu_t)_{t \in [0, T]}$ satisfies $\nu_t\neq\delta_0$ for every $t\in(0,T_0]$.
  \end{enumerate}
\end{theorem}

In particular, $X\equiv0$ and $Y^{*}$ have different laws on $[0,T_0]$, so uniqueness in law fails for the McKean--Vlasov SDE~\eqref{eq:MV} with coefficients $b,\sigma$ given in~\eqref{eq:McKV_counterexample}. Since $X$ is trivial, the solutions can be constructed as strong solutions driven by the same Brownian motion. Hence, Theorem~\ref{thm:main} directly implies that pathwise uniqueness fails for McKean--Vlasov SDE~\eqref{eq:MV} with coefficients $b,\sigma$ given in~\eqref{eq:McKV_counterexample} without invoking Yamada--Watanabe theorem. 

\smallskip

We conclude the section with some exemplary diffusion coefficients satisfying the assumptions of Theorem~\ref{thm:main}.

\begin{example}
  The assumptions of Theorem~\ref{thm:main} are satisfied by a broad class of functionals on $\mathcal{P}_1(\mathbb{R})$. The simplest example is
  \begin{equation*}
    F(\mu) := a\int_{\mathbb{R}} |x|\,\mu(\d x) = a \|\mu\|_1, \qquad a>0.
  \end{equation*}
  More generally, let $g\colon \mathbb{R}\to[0,\infty)$ be Lipschitz continuous with $g(0)=0$ and
  \begin{equation*}
    g(x)\geq \kappa |x|, \qquad x\in\mathbb{R},
  \end{equation*}
  for some $\kappa>0$. Then
  \begin{equation*}
    F_g(\mu) := \int_{\mathbb{R}} g(x)\,\mu(\d x)
  \end{equation*}
  satisfies the assumptions of Theorem~\ref{thm:main}. Indeed,
  \begin{equation*}
    |F_g(\mu)-F_g(\nu)| \leq \operatorname{Lip}(g) W_1(\mu,\nu)
  \end{equation*}
  and
  \begin{equation*}
    F_g(\mu) \geq \kappa \int_{\mathbb{R}} |x|\,\mu(\d x) = \kappa \|\mu\|_1.
  \end{equation*}
  A more general class is obtained by setting
  \begin{equation*}
    F(\mu) := a W_1(\mu,\delta_0) + G(\mu),
  \end{equation*}
  where $a>0$ and $G\colon\mathcal{P}_1(\mathbb{R})\to[0,\infty)$ is $W_1$-Lipschitz continuous with $G(\delta_0)=0$. Then
  \begin{equation*}
    F(\mu)\geq a\|\mu\|_1,
  \end{equation*}
  and $F$ is $W_1$-Lipschitz continuous with Lipschitz constant at most $a+\operatorname{Lip}(G)$. For instance, choosing
  \begin{equation*}
    G(\mu) := \lambda  \left| \int_{\mathbb{R}} x\,\mu(\d x) \right|, \qquad \lambda\geq 0,
  \end{equation*}
  gives
  \begin{equation*}
    F(\mu) = a\int_{\mathbb{R}} |x|\,\mu(\d x) + \lambda \left| \int_{\mathbb{R}} x\,\mu(\d x) \right|.
  \end{equation*}
\end{example}

\section{Construction of a nontrivial solution}\label{sec:proof}

In this section, we prove Theorem~\ref{thm:main}, with the main effort devoted to the construction of the nontrivial solution asserted in Theorem~\ref{thm:main}~(iii). The key idea is to obtain such a solution through a fixed-point argument for a family of one-dimensional SDEs in which the law dependence of the diffusion coefficient is replaced by a fixed law. The main probabilistic ingredient is a quantitative spreading estimate, showing that solutions of the resulting ordinary stochastic differential equations develop a strictly positive first moment, uniformly with respect to the initial condition; see Subsection~\ref{subsec:spreading}. This estimate is combined with an a priori upper bound to construct an invariant set for the associated fixed-point map. Schauder's fixed-point theorem then yields a nontrivial solution of the McKean--Vlasov SDE; see Subsection~\ref{subsec:fixedpoint}. The proof of Theorem~\ref{thm:main} is completed in Subsection~\ref{subsec: proof of main result}. We first introduce the notation and parameters used throughout this section.

\subsection{Setting and notation}

Throughout, for $\alpha\in(\tfrac12,1)$ we write
\begin{equation*}
  \beta:=2-\frac1\alpha\in(0,1),\qquad
  \gamma:=\frac{1}{1-\beta}=\frac{\alpha}{1-\alpha}\in(1,\infty),\qquad
  J_\alpha:=\int_0^\infty\frac{1}{(1+\eta^{\alpha})^{2}}\dd\eta<\infty ,
\end{equation*}
the finiteness of $J_\alpha$ being exactly the condition $\alpha>\frac12$. The quantities blow up or degenerate as $\alpha\downarrow\frac12$ ($\beta\downarrow0$, $J_\alpha\uparrow\infty$) and as $\alpha\uparrow1$ ($\gamma\uparrow\infty$). Note for later reference the elementary identity
\begin{equation}\label{eq:gamma-identity}
  \gamma\beta+1=\gamma.
\end{equation}

\smallskip

For $E\subset \R^m$, $m \in \N$, and $A \subset \R$, let $C(E;A)$ denote the space of continuous functions from $E$ to $A$. We write $C(E):=C(E; \R)$. Moreover, we denote the subspace of bounded continuous functions by $C_b(E) \subset C(E)$ and the subspace of twice continuously differentiable functions with compact support by $C_0^2(\R)\subset C(\R)$.
 
\subsection{A quantitative spreading estimate}\label{subsec:spreading}

We begin with the key probabilistic estimate: a uniformly positive additive contribution to the diffusion coefficient forces the solution to spread away from the origin at a quantitative rate, independently of its initial position.
 
\begin{lemma}\label{lem:spreading}
  Let $\alpha\in(\tfrac12,1]$, $C_1>0$, $m>0$, and set $\bar a(y):=C_1|y|^{\alpha}+m$ and $S_m:=\int_0^\infty\bar a(z)^{-2}\dd z$. Let $0\le t_0<t_1\le T$, $h:=t_1-t_0$, and let $(Y_s)_{s\in[t_0,t_1]}$ be a continuous adapted process with
  \begin{equation*} 
    Y_s=Y_{t_0}+\int_{t_0}^{s}H_u\,\dd B_u,\qquad s\in[t_0,t_1],
  \end{equation*}
  for some progressively measurable $H$ with 
  \begin{equation}\label{eq:integrability_a}
    \E\Bigg[\int_{t_0}^{t_1}H_u^2 \,\dd u \Bigg]+ \E \big[ Y_{t_0}^2 \big] <\infty
  \end{equation}
  and
  \begin{equation}\label{eq:domination}
    |H_u|\;\ge\;\bar a(Y_u)\qquad\text{for a.e.\ }u\in[t_0,t_1],\ \text{a.s.}
  \end{equation}
  Then, $S_m = J_\alpha C_1^{-1/\alpha}m^{-\beta}$ and whatever the possibly random initial value $Y_{t_0}$,
  \begin{equation*}
    \E[|Y_{t_1}|]\;\ge\;\frac{h}{16\,S_m}\;=\;\frac{C_1^{1/\alpha}}{16\,J_\alpha}\,m^{\beta}\,h .
  \end{equation*}
\end{lemma}
 
\begin{proof}
  \emph{Step 0 (speed measure).} We derive the formula for the \(S_m\). For $z=(m/C_1)^{1/\alpha}\eta$, we have
  \begin{equation*}
    \int_0^\infty\frac{1}{(C_1z^{\alpha}+m)^{2}}\dd z
    =\Bigl(\frac{m}{C_1}\Bigr)^{1/\alpha}\frac1{m^{2}}\int_0^\infty\frac{1}{(1+\eta^{\alpha})^{2}} \dd\eta
    =J_\alpha\,C_1^{-1/\alpha}\,m^{1/\alpha-2},
  \end{equation*}
  and $1/\alpha-2=-\beta$. The integral converges at infinity because $2\alpha>1$.
 
  \emph{Step 1 (exit-time bound).} We set
  \begin{equation*}
    w:=\frac{h}{8S_m},\qquad
    \tau:=\inf\{s\ge t_0:\ |Y_s|\ge w\}\wedge t_1 \quad(\tau= t_0\text{ if }|Y_{t_0}|\ge w).
  \end{equation*}
  and define, for $y\in[-w,w]$,
  \begin{equation*}
    V(y):=\int_{-w}^{w}G(y,z)\,\frac{2}{\bar a(z)^{2}}\,\dd z,
    \qquad G(y,z):=\frac{(w+(y\wedge z))(w-(y\vee z))}{2w},
  \end{equation*}
  Note that $G$ is the Green function for $-\d^2/\d y^2$ on $(-w,w)$ with Dirichlet boundary conditions. Consequently $V$ satisfies \(V(w) = V(-w) =0\) and 
  \begin{equation}\label{eq:ode}
    V''(y) = -\frac{2}{\bar a^2(y) }, \qquad y \in (-w,w).
  \end{equation}
  Then, $V\in C^{2}([-w,w])$ with $V''(y) \le 0$ and $V\ge0$, and $V'$ is bounded ($|\partial_yG|\le1$ gives $|V'|\le\int_{-w}^{w}2\bar a(y)^{-2} \dd y \le4S_m$), and, using $G(y,z)\le w-|z|$ (for $z \ge 0$, $z\ge y$, we have $G=(w+y)(w-z)/(2w)\le w-z$; symmetrically for $z\le y$ and similar for \(z \le 0\)), we find
  \begin{equation}\label{eq:Vmax}
    0\le V(y)\le\int_{-w}^{w}(w-|z|)\frac{2}{\bar a(z)^2}\dd z \le 2w\int_{-w}^{w}\frac{1}{\bar a(z)^2} \dd z \le 4wS_m .
  \end{equation}
  In the event $\{|Y_{t_0}|<w\}$, we utilize It\^o's formula on $[t_0,\tau]$ to obtain 
  \begin{equation*}
    V(Y_{\tau})=V(Y_{t_0})+\int_{t_0}^{\tau}V'(Y_u)H_u\dd B_u
    +\frac12\int_{t_0}^{\tau}H_u^{2}\,V''(Y_u)\,\dd u.
  \end{equation*}
  Now, since \(V\) solves the ODE~\eqref{eq:ode} and by \eqref{eq:domination}, we have
  \begin{equation*} 
    \frac12 H_u^{2}V''(Y_u)=-\Bigl(\frac{H_u}{\bar a(Y_u)}\Bigr)^{2}\le-1
    \qquad\text{for a.e.\ }u \in [t_0,\tau] .
  \end{equation*}
  Using this and the first inequality in \eqref{eq:Vmax} we obtain on the event $\{|Y_{t_0}|<w\}$, 
  \begin{equation*}
    \tau - t_0 \leq V(Y_{t_0})+\int_{t_0}^{\tau}V'(Y_u)H_u\dd B_u.
  \end{equation*}
  By $V'$ being bounded and \eqref{eq:integrability_a} the stochastic integral is a true martingale, so taking expectations we get:
  \begin{equation*}
    \E\bigl[(\tau-t_0)\mathbf 1_{\{|Y_{t_0}|<w\}}\bigr]\le\E[V(Y_{t_0}) \mathbf 1_{\{|Y_{t_0}|<w\}}]\le4wS_m .
  \end{equation*}
  On $\{|Y_{t_0}|\ge w\}$, $\tau-t_0=0$. 
 
  Hence $\E[\tau-t_0]\le4wS_m$, and by Markov's inequality and the choice of $w$,
  \begin{equation*}
    \PP\bigl(\{\tau=t_1 \}\cap \{ |Y_{\tau}|<w\}\bigr)\le\PP(\tau-t_0\ge h)\le\frac{4wS_m}{h}=\frac12 .
  \end{equation*}
 
  \emph{Step 2 (conclusion).} By \eqref{eq:integrability_a} $Y$ is a true martingale on $[t_0,t_1]$, so $|Y|$ is a submartingale; by Doob's optional sampling theorem with the bounded stopping times $\tau\le t_1$,
  \begin{equation*} 
    \E[|Y_{t_1}|]\ \ge\ \E[|Y_{\tau}|]\\ \ge\ \E\bigl[|Y_\tau|\mathbf 1_{\{|Y_\tau|\ge w \} }\bigr]
    \ \ge\ w\,\PP\bigl(|Y_\tau|\ge w\bigr)\ \ge\ \frac w2=\frac{h}{16S_m},
  \end{equation*}
  where the last inequality follows from 
  \begin{equation*}
    \{\tau=t_1 \} \cap \{ |Y_{\tau}| < w\} \subset  \{ |Y_{\tau}| <  w\}
  \end{equation*}
  since, on $\{t_0 < \tau< t_1\}$ we have by continuity \(|Y_\tau| = w \). For $\tau=t_0$ the inclusion is clear.
\end{proof}

\subsection{A fixed-point construction}\label{subsec:fixedpoint}

Throughout this section fix $\alpha\in(\frac12,1)$, $C_1>0$. Moreover, define $\Lambda := \max\{1,L_F\}$ and the constants
\begin{align}\label{eq:constants}
  &M:=\max\Bigl(1,\;\bigl(8C_1^{2}\bigr)^{\frac1{1-\alpha}}\Bigr),\qquad
  T_0:=\min\Bigl(T,\;\tfrac{1}{4\Lambda^2}\Bigr),\\
  &\varepsilon_0:=\Bigl(\frac{\kappa C_1^{1/\alpha}}{16\,J_\alpha\,2^{\gamma}}\Bigr)^{\gamma},
  \qquad
  \varepsilon:=\min(\varepsilon_0,1).\label{eq:eps0}
\end{align}

For fixed $r\in C([0,T_0];\;[0,\infty))$, we consider the solution $Y^{(r)}$ of the \emph{frozen} SDE
\begin{equation}\label{eq:frozen}
  \dd Y_t=\bigl(C_1|Y_t|^{\alpha}+r(t)\bigr)\dd B_t,\qquad Y_0=0 .
\end{equation}
Existence of a strong solution and pathwise uniqueness of the SDE~\eqref{eq:frozen} are granted in this frozen setting by classical results about ordinary SDEs, see e.g. \cite[Chapter~IX, Theorem~3.5]{Revuz1999} and \cite[Theorem~2.3]{Promel2023}, since the coefficient $a_r(t,y):=C_1|y|^{\alpha}+r(t)$ is continuous, of linear growth ($|y|^\alpha\le1+|y|$), and $\alpha$-H\"older continuous in $y$ uniformly in $t$. Hence, we can define
\begin{equation*}
  \Phi(r)(t):= F(\cL(Y^{(r)}_t)),\qquad t\in[0,T_0].
\end{equation*}
 
\begin{lemma}\label{lem:envelope}
  Let $r\in C([0,T_0];\;[0,\infty))$ and \(Y^{(r)}\) be the strong solution of~\eqref{eq:frozen}. Let \(r\) satisfy $r(t)\le \Lambda \sqrt{Mt}$ for $t \in [0,T_0]$. 
   Then the following hold, with constants independent of $r$:
  \begin{enumerate}[label=\textup{(\roman*)}]
    \item $Y$ is a square-integrable martingale on $[0,T_0]$ and
          $\E[Y_t^2]\le Mt$ for all $t\in[0,T_0]$. In particular $\E[|Y_t|]\le\sqrt{Mt}$.
    \item $\sup_{t\in[0,T_0]}\E[Y_t^4]\le C_4
          :=\tfrac38\bigl(C_1^2+\tfrac M4\bigr)M\,e^{3C_1^2}$,
          which depends only on $\alpha$ and $C_1$.
    \item $0\le\Phi(r)(t)\le L_F\sqrt{Mt}\le\Lambda\sqrt{Mt}$ for all $t\in[0,T_0]$.
  \end{enumerate}
\end{lemma}

\begin{proof}
By linear growth and a standard localised Gr\"onwall bound, $\E[\int_0^{T_0}a_r(t,Y_t)^2\dd t]<\infty$. Hence, $Y$ is a true $L^2$-martingale and $\varphi(t):= \E[Y_t^{(r)}]=\int_0^t\E\bigl[a_r(s,Y_s)^2\bigr]\dd s$ is continuous, $\varphi(0)=0$. Using $(x+y)^2\le2x^2+2y^2$, Jensen ($z\mapsto z^{\alpha}$ concave), and $r(s)^2\le \Lambda^2 Ms$ we find 
  \begin{equation}\label{eq:phi-int}
    \varphi(t) = \int\limits_0^t \E \bigl [ (C_1 |Y_s^{(r)}|^\alpha + r(s))^2 \bigr] \dd s \le\int_0^t (2C_1^{2}\varphi(s)^{\alpha}+2 \Lambda^2 Ms)\dd s .
  \end{equation}
  Let \(A:=\bigl\{t\in[0,T_0]\colon \varphi(s)\le Ms \text{ for all } s\in[0,t]\bigr\}.\)
  We show that $A$ is nonempty, relatively closed and relatively open in $[0,T_0]$; since $[0,T_0]$ is connected, this gives $A=[0,T_0]$, which is the claim. Clearly $0\in A$.

  \emph{Key estimate.} Let $t\in A$. Inserting $\varphi(s)\le Ms$ for $s\le t$ into~\eqref{eq:phi-int} yields
  \begin{equation} \label{eq: key}
    \varphi(t)\le 2C_1^{2}M^{\alpha}\frac{t^{1+\alpha}}{1+\alpha}+\Lambda^2 Mt^{2}
    \le t\bigl(2C_1^{2}M^{\alpha}T_0^{\alpha}+\Lambda^2 M T_0\bigr)
    \le \frac{M}{2}\,t .
  \end{equation}
  For the last inequality, note that $T_0\le\frac1{4\Lambda^2}\le\frac14$ by~\eqref{eq:constants}. Hence $\Lambda^2MT_0\le\frac M4$. Moreover, $M^{1-\alpha}\ge 8C_1^2\ge 8C_1^2T_0^\alpha$, i.e.\ $2C_1^2M^\alpha T_0^\alpha\le\frac M4$.

  \emph{$A$ is closed.} Since $\varphi$ is continuous, $A=\{t\in[0,T_0]\colon \sup_{s\in[0,t]}(\varphi(s)-Ms)\le 0\}$ is the sublevel set of a continuous function of $t$.

  \emph{$A$ is open at every $t\in A\setminus\{0\}$.} By estimate~\eqref{eq: key}, $\varphi(t)\le\frac M2t<Mt$. By continuity of $\varphi$ there is $\delta>0$ with $\varphi(s)<Ms$ for all $s\in[t,t+\delta)\cap[0,T_0]$. Since also $t\in A$, we get $[0,t+\delta)\cap[0,T_0]\subseteq A$.

  \emph{$A$ is open at $0$.} Here the key estimate gives nothing, since $\varphi(0)=0=M\cdot0$. Instead we use $\varphi(0)=0$ and the continuity of $\varphi$ to choose $\delta>0$ such that $2C_1^2\varphi(s)^\alpha+2\Lambda^2Ms\le M$ for all $s\in[0,\delta]$.
  Then~\eqref{eq:phi-int} gives $\varphi(t)\le Mt$ for all $t\in[0,\delta]$, i.e.\ $[0,\delta]\subseteq A$. This proves (i). 
  The claim (ii) follows by linear growth and a Gr\"onwall argument. Furthermore, by the Lipschitz property of \(F\) we find $\Phi(r)(t) \leq L_F \sqrt{Mt} \leq \Lambda \sqrt{Mt}$, which demonstrates (iii).  
\end{proof}

Define $K:=\bigl\{r\in C([0,T_0])\colon \varepsilon t^{\gamma}\le r(t)\le \Lambda \sqrt{Mt} \ \text{ for all } t\in[0,T_0]\bigr\}.$

\begin{proposition}[The fixed point]\label{prop:fixedpoint}
  There exists $r^{*}\in K$ with $\Phi(r^{*})=r^{*}$.
\end{proposition}

\begin{proof}
  We verify the hypotheses of Schauder's fixed-point theorem, see, e.g. \cite[Theorem~8.8]{Deimling1985}. First, $K$ is nonempty. Notice, that $t\mapsto\varepsilon t^{\gamma}$ lies in $K$: $\varepsilon t^{\gamma}\le\varepsilon t^{1/2}\le\sqrt{Mt}$ on $[0,T_0]$ since $\gamma>\frac12$, $T_0\le1$, $\varepsilon\le1\le\sqrt M$. Convexity, closedness, and boundedness are clear (pointwise two-sided constraints, continuous bounds).

  Next we verify that $\Phi$ maps $K$ into $K$:
 
  \emph{Upper bound.} For $r\in K$, Lemma~\ref{lem:envelope} gives $\Phi(r)(t)\le \Lambda \sqrt{Mt}$.
 
  \emph{Lower bound (the barrier propagates).} Fix $t\in(0,T_0]$ and $r\in K$. For all $s \in [t/2,t]$,
  \begin{equation*}
    a_r(s,y)=C_1|y|^{\alpha}+r(s)\ \ge\ C_1|y|^{\alpha}+ \epsilon \frac{t^{\gamma}}{2^\gamma}
  \end{equation*}
  because $r(s)\ge\varepsilon s^{\gamma}\ge\varepsilon(t/2)^{\gamma}$ for $s\ge t/2$. Lemma~\ref{lem:spreading} with $H_s=a_r(s,Y^{(r)}_s)$ on $[t/2,t]$ (with $h=t/2$, arbitrary initial value $Y_{t/2}$) yields, using \eqref{eq:gamma-identity},
  \begin{align*}
    \Phi(r)(t)&= F(\cL(Y^{(r)}_t))\geq \kappa \E[|Y^{(r)}_t|] \;\ge\; \kappa\frac{C_1^{1/\alpha}}{16J_\alpha}\, \epsilon^\beta \frac{t^{\beta\gamma}}{2^{\beta\gamma}}\,\frac t2 \\
    &=\; \kappa \frac{C_1^{1/\alpha}\varepsilon^{\beta}}{16J_\alpha\,2^{\gamma\beta+1}}\;t^{\,\gamma\beta+1} \;=\;  \kappa \frac{C_1^{1/\alpha}}{16J_\alpha 2^{\gamma}}\;\varepsilon^{\beta}\,t^{\gamma} \;\ge\;\varepsilon\,t^{\gamma},
  \end{align*}
  the last inequality being exactly $\varepsilon^{1-\beta}\le \kappa C_1^{1/\alpha}/(16J_\alpha2^{\gamma})$, i.e.\ $\varepsilon\le\varepsilon_0$ by \eqref{eq:eps0} and $\frac1{1-\beta}=\gamma$.
  At $t=0$, we have $\Phi(r)(0)=0=\varepsilon\cdot0$. Hence $\Phi(K)\subseteq K$.
 
  $\Phi(K)$ is \emph{relatively compact}: For $r\in K$ and $0\le t\le t'\le T_0$,
  \begin{align*}
    \bigl|\Phi(r)(t')-\Phi(r)(t)\bigr|
    &=|F(\cL(Y^{(r)}_{t'}))-F(\cL(Y^{(r)}_t))| 
    \leq L_F \W_1(\cL(Y^{(r)}_{t'}),\cL(Y^{(r)}_t)) \\
    &\le L_F \E \bigl[ \bigl|Y^{(r)}_{t'}-Y^{(r)}_{t}\bigr| \bigr] 
    \le L_F \Bigl(\int_t^{t'}\E\bigl[ a_r(s,Y^{(r)}_s)^2 \bigr] \dd s\Bigr)^{1/2}
    \le L_F C \,|t'-t|^{1/2},
  \end{align*}
  where we used $(a+b)^2 \leq 2a^2 + 2 b^2$ and the Jensen inequality $\E[|Y_s^{(r)}|^{2 \alpha}]\leq \E[|Y_s^{(r)}|^{2}]^{ \alpha}$, the uniform bounds from Lemma~\ref{lem:envelope} and the definition of K in the last step. Uniform $\frac12$-H\"older equicontinuity plus uniform boundedness give relative compactness in $C([0,T_0])$ by the Arzel\`a--Ascoli theorem.
 
  $\Phi:K\to K$ is \emph{continuous}: Let $r_n\to r$ uniformly, $r_n,r\in K$, with solutions $Y^n:=Y^{(r_n)}$, $Y:=Y^{(r)}$. By Lemma~\ref{lem:envelope} and Burkholder--Davis--Gundy inequality,
  \begin{equation*}
    \E \bigl[ |Y^n_{t'}-Y^n_t|^{4}\bigr] \le c\,\E\bigg[ \bigg( \int_t^{t'}a_{r_n}(s,Y^n_s)^2 \dd s\bigg)^{2}\bigg] \le c\,(t'-t)\int_t^{t'}\E\bigl[a_{r_n}(s,Y^n_s)^4\bigr] \dd s\le c'(t'-t)^2
  \end{equation*}
  uniformly in $n$ (fourth moments uniformly bounded), so that $\{\cL(Y^n);\;n \in \N\}$ is tight in $C([0,T_0])$ by Kolmogorov's criterion (see \cite[Problem~2.4.11]{Karatzas1991}). By Prokhorov's theorem (\cite[Theorem~2.4.7]{Karatzas1991}) we obtain the relative compactness (\cite[Definition~2.4.6]{Karatzas1991}) of $\{\cL(Y^n);\;n \in \N\}$ in the space of probability measures on $C([0,T_0])$. That is, we can find a converging subsequence $\{\cL(Y^{n_k});\;k \in \N\}$ with $\cL(Y^{n_k})$ converging to some $Q\in \cP(C([0,T_0]))$ weakly as $k \to \infty$. By the Skorokhod representation theorem (see e.g. \cite[Theorem~11.7.2]{Dudley2002}) there is a probability space $(\hat{\Omega}, \hat{\mathcal{F}}, \hat{\bP})$ with continuous stochastic processes $\bar{Y}^k$, $k \in \N$, and $\hat{Y}$ on it, such that $(\bar{Y}^k) \stackrel{\mathscr{D}}{\sim} Y^{n_k}$, $k \in \N$, $\hat{Y} \sim Q$, and $\bar{Y}^k \to \hat{Y}$ in $C([0,T_0])$ $\hat{\P}$-a.s. as $k \to \infty$. Since $\bar{Y}^k_0=0$ for all $k \in \N$, we directly get $\hat{Y}_0=0$. 

  Let us define $\mathcal{A}, \mathcal M, \mathcal{A}^k, \mathcal M^{k}$, $k \in \N$, by
  \begin{align*}
    &(\mathcal{A}^k_t f)(y)= \tfrac{1}{2}(C_1 |y_t|^\alpha + r_{n_k}(t))^2 f''(y(t)), \\
    &(\mathcal{A}_t f)(y)= \tfrac{1}{2}(C_1 |y_t|^\alpha + r(t))^2 f''(y(t)), \\
    &\mathcal M^{k,f}_t:=f(\bar{Y}^k_t)-f(\bar{Y}^k_0)-\int_0^t (\mathcal A^k_s f)(\bar{Y}^k) \dd s, \\
    &\mathcal M^f_t:=f(\hat{Y}_t)-f(\hat{Y}_0)-\int_0^t (\mathcal A_s f)(\hat{Y}) \dd s,
  \end{align*}
  for $t \in [0,T_0], \; f \in C_0^2(\R), y \in C([0,T_0])$, and $k \in \N$.
  
  Since $Y^n$, $(\Omega,\mathcal F,\PP)$, $(\mathcal F_t)_{t \in[0, T]}$ are weak solutions to \eqref{eq:frozen} w.r.t $r_n$ for all $n \in \N$ and since $\cL(\bar{Y}^k)=\cL(Y^{n_k})$, by \cite[Corollary~5.4.8]{Karatzas1991}, for every $f \in C_0^2(\R)$, $\mathcal M^{k,f}$ is a local martingale with respect to the natural filtration of $\bar{Y}^k$. Since $f \in C_0^2(\R)$ and the coefficients are bounded on the support of $f'$ and $f''$, $\mathcal M^{k,f}$ is a true martingale.

  For $0 \leq t_1,\ldots,t_m \leq s \leq t \leq T_0$, and $G\in C_b(\R^m)$ we have
  \begin{equation*}
    \hat{\E}[(\mathcal M^{k,f}_t-\mathcal M^{k,f}_s)G(\bar{Y}^k_{t_1},\ldots,\bar{Y}^k_{t_m})] = 0.
  \end{equation*}
  By the uniform convergence $r_{n_k} \to r$ and $\bar Y^{k} \to \hat Y$, and by $f''$ having compact support, dominated convergence yields the corresponding identity for $\hat{Y}$. A monotone-class argument shows that $\mathcal M^f$ is a martingale with respect to the natural filtration of $\hat{Y}$. By localization, $Q=\cL(\hat{Y})$ solves the local martingale problem for \eqref{eq:frozen} in the sense of \cite[Definition~5.4.5]{Karatzas1991} and by \cite[Corollary~5.4.8]{Karatzas1991} $Q$ is the law of a weak solution of \eqref{eq:frozen}.

  \smallskip

  Pathwise uniqueness for \eqref{eq:frozen} implies uniqueness in law, see Yamada--Watanabe theorem in, e.g., \cite[Proposition~5.3.20]{Karatzas1991}, so $Q = \cL(Y)$ and the full sequence converges, i.e. $\cL(Y^n)\to\cL(Y)$ weakly, i.e. for every subsequence we can find a further subsequence, which converges to the same limit. Since the one-dimensional marginals converge and $\{|Y^n_t|\}_{n \in \N}$ is uniformly integrable (fourth moments), we get that $\W_1(\cL(Y^n_t),\cL(Y_t))\to 0$ for all $t\in[0,T_0]$. Therefore, $|\Phi(r_n)(t)-\Phi(r)(t)| \leq L_F \W_1(\cL(Y^n_t),\cL(Y_t)) \to 0$ for every $t\in[0,T_0]$; combined with the uniform $\frac12$-H\"older equicontinuity, the convergence is uniform on $[0,T_0]$. Hence, $\Phi$ is continuous, and Schauder's fixed-point theorem applies.
\end{proof}

\subsection{Proof of Theorem~\ref{thm:main}}\label{subsec: proof of main result}

With the preparations of the previous subsection in place, we are now in a position to prove Theorem~\ref{thm:main}.

\begin{proof}[Proof of Theorem~\ref{thm:main}]
  \emph{(i) Hypotheses.} Space regularity: By the $\alpha$-H\"older continuity of $z\mapsto z^\alpha$ on $[0,\infty)$ we get $|\sigma(t,x,\mu)-\sigma(t,y,\mu)|\le C_1\bigl||x|^{\alpha}-|y|^{\alpha}\bigr| \le C_1|x-y|^{\alpha}$. 

  Measure regularity: $F$ is $L_F$-Lipschitz from $(\cP_1,\W_1)$ to $\R_+$ by definition and for every $\mu,\nu \in \cP_p(\R)$ we have by H\"older's inequality $|\sigma(t,x,\mu)-\sigma(t,x,\nu)| \leq L_F \W_1(\mu,\nu)\leq L_F \W_p(\mu,\nu)$.

  Linear growth: For all $\mu \in \cP_1(\R)$ we have $F(\mu)\leq L_F \W_1(\mu, \delta_0) \leq L_F \|\mu\|_1$, and hence $\sigma(t,x,\mu)\le C_1(1+|x|)+L_F \|\mu\|_1$ for all $x \in \R$, $\mu \in \cP_1(\R)$. Thus \ref{A2} holds with $C_\sigma=\max\{C_1,L_F\}$.
 
  \emph{(ii).} With $X\equiv0$: $\mu_t=\delta_0$, $\sigma(t,X_t,\mu_t) =\sigma(t,0,\delta_0)=0$, so $X_t=0=\int_0^t0\,\dd B$ solves \eqref{eq:MV}.
 
  \emph{(iii).} Apply Proposition~\ref{prop:fixedpoint}. There is a $r^{*}\in K$ with $\Phi(r^{*})=r^{*}$. Let $Y^{\text{loc}}:=Y^{(r^{*})}$ be the strong solution of \eqref{eq:frozen} and $\nu_t:=\cL(Y^{\text{loc}}_t)$. Then
  \begin{equation*}
    F(\nu_t) =\Phi(r^{*})(t)=r^{*}(t),
  \end{equation*}
  so the coefficient actually used in \eqref{eq:frozen} is $C_1|y|^{\alpha}+r^{*}(t)=C_1|y|^{\alpha}+F(\nu_t)=\sigma(t,y,\nu_t)$. Hence, $Y^{\text{loc}}$ solves the McKean--Vlasov equation \eqref{eq:MV} on $[0,T_0]$ with its own law flow, started from $\delta_0$. Membership $r^{*}\in K$ gives 
  \begin{equation*}
    \E[|Y^{\text{loc}}_t|] = \W_1(\nu_t,\delta_0) \geq \frac{F(\nu_t)}{L_F}= \frac{r^*(t)}{L_F}\ge \frac{\varepsilon t^{\gamma}}{L_F},
  \end{equation*}
  with $\varepsilon=\min(\varepsilon_0,1)$, $\varepsilon_0$ as in \eqref{eq:eps0}. By Lemma~\ref{lem:envelope} we have $\E[|Y^{\text{loc}}_{T_0}|^4]<\infty$. Hence, $\nu_{T_0} \in \cP_2(\R)$, so if $T_0<T$ we can use Remark~\ref{prop:existence} and obtain a weak solution of a time-shifted SDE on $[T_0,T]$ with initial condition $\cL(Z_{T_0})=\nu_{T_0}$. Denote the solution process by $Z$ and its deterministic law flow by $m_t:=\cL(Z_t)$, $t \in [T_0,T]$. With this frozen law we consider on the original probability space the SDE
  \begin{equation}\label{eq:new_SDE}
      \d \tilde{Y}_t = (C_1 |\tilde{Y}_t|^\alpha + F(m_t)) \dd B_t, \qquad \tilde{Y}_{T_0}=Y^{\text{loc}}_{T_0}.
  \end{equation}
  Note, that this is an ordinary one-dimensional SDE with its coefficient being a deterministic function of time and space of linear growth, that is $\alpha$-H\"older in space with $\alpha > 1/2$. Therefore, pathwise uniqueness holds for this SDE. Since $Z$ provides weak existence, by the Yamada--Watanabe theorem there is a strong solution $\tilde{Y}$ with uniqueness in law. Therefore, we have $\cL(\tilde{Y}_t) = \cL(Z_t) = m_t$. Substituting this equality into \eqref{eq:new_SDE} we see that $\tilde{Y}$ solves 
  \begin{equation*}
      \d \tilde{Y}_t = (C_1 |\tilde{Y}_t|^\alpha + F(\cL(\tilde{Y}_t))) \dd B_t, \qquad \tilde{Y}_{T_0}=Y^{\text{loc}}_{T_0},
  \end{equation*}
  i.e. $\tilde{Y}$ is itself a solution to the aforementioned McKean--Vlasov SDE. The concatenation $Y^\ast$ of $Y^{\text{loc}}$ and $\tilde{Y}$ then is a strong solution of \eqref{eq:MV} with the choice of \eqref{eq:McKV_counterexample} on $[0,T]$.
\end{proof}

\appendix
\section{Weak existence under Assumption~\ref{ass:main}}\label{sec:appendix}

In this appendix, we prove the existence of weak solutions to the McKean--Vlasov SDE~\eqref{eq:MV} under Assumption~\ref{ass:main}, as stated in Remark~\ref{prop:existence}.

\begin{proof}[Proof of Remark~\ref{prop:existence}]
  We shall check Assumptions~2.2, 2.3, 2.5, 2.6, 2.7 presented in \cite{Hammersley2021} with $v(t,x,\mu)=1+x^2$ and $D_k=(-k,k)$ for all $k \in \N$ in order to apply \cite[Theorem~2.10]{Hammersley2021}. To that end, we consider the diffusion generator $L$, defined as
  \begin{equation*}
    (L v)(t,x,\mu):= 2xb(t,x,\mu)+\sigma(t,x,\mu)^2.
  \end{equation*}

  For \cite[Assumption~2.2 (i)]{Hammersley2021}, using the linear growth condition in Assumption~\ref{ass:main}, we get for every $k \in \N$ and $\mu \in \mathcal{P}(D_k)$ that
  \begin{equation*}
    \int_{D_k}L(t,x,\mu)v(t,x,\mu)\,\mu(\d x)
    = \int_{D_k} 2xb(t,x,\mu)+\sigma(t,x,\mu)^2\, \mu(\d x)
    \leq C \int_{D_k}v(t,x,\mu)\,\mu(\d x).
  \end{equation*}
  \cite[Assumption~2.2 (ii)]{Hammersley2021} clearly holds with the choice $V(t,x) = v(t,x,\mu)$.

  \cite[Assumption~2.3]{Hammersley2021} can be easily checked with the choice
  \begin{equation*}
    \mu_0^k = \mathbbm{1}_{\{|x|< k\}} \mu_0 + \mu_0(\{|x| \geq k\}) \delta_0,\quad  k \in \N.
  \end{equation*}

  Any sequence $(\mu_n)_{n \in \N}$ satisfying the conditions of \cite[Assumption~2.5]{Hammersley2021} with $v$ as chosen above must have uniformly integrable first moments. Hence, $\mu_n \Rightarrow \mu$ in the sense of weak convergence implies $\W_1(\mu_n,\mu)\to 0$ as $n \to \infty$. Now the continuity conditions in Assumption~\ref{ass:main} give \cite[Assumption~2.5]{Hammersley2021}.

  \cite[Assumption~2.6]{Hammersley2021}: On $|x| \leq k$ we have, by the linear growth condition in Assumption~\ref{ass:main}, that
  \begin{equation*}
    |b(t,x,\mu)|+|\sigma(t,x,\mu)| \leq C(1+k+\|\mu\|_1) \leq C_k(1+\|\mu\|_1).
  \end{equation*}

  By moment estimates similar to those used to show \cite[Assumption~2.2 (i)]{Hammersley2021}, one can deduce that
  \begin{equation*}
    \int_\R |b(t,x,\mu)| + |\sigma(t,x,\mu)|^2 \mu(\d x) \leq C \int_\R v(t,x,\mu)\,\mu(\d x),
  \end{equation*}
  giving \cite[Assumption~2.7]{Hammersley2021}.

  Finally, we need to check for $\theta(t)=\exp(- \int_0^t m_1(s)\dd s)$ that
  \begin{equation*}
    \mathcal M(t) = \frac{\varphi_v(\int_\R v(0,x,\mu_0) \mu_0(\d x))}{\theta(t)} + \int_0^t \frac{\theta(s)}{\theta(t)} m_2(s) \dd s
  \end{equation*}
  satisfies $\sup_{t \in [0,T]} \mathcal M(t) < \infty$, where $\varphi_v$, $m_1$, $m_2$ can be chosen depending on the other conditions. Choosing $\varphi_v= id$, $m_1=C$ and $m_2=0$, we get
  \begin{equation*}
    \mathcal M(t) = \frac{\int_\R v(0,x,\mu_0)\, \mu_0(\d x)}{\theta(t)}
  \end{equation*}
  and the claimed inequality holds by the boundedness of the denominator.

  Then, \cite[Theorem~2.10]{Hammersley2021} gives the existence of a weak solution $X$ to the McKean--Vlasov SDE~\eqref{eq:MV} satisfying
  \begin{equation*}
    \sup_{t \in [0,T]} \E[v(t,X_t,\cL(X_t))] = \sup_{t \in [0,T]} 1+ \E[X_t^2] < \infty.
  \end{equation*}
\end{proof}

\bibliography{literature}{}
\bibliographystyle{amsalpha}

\end{document}